\documentclass[12pt,leqno]{amsart}\usepackage[latin9]{inputenc}\usepackage{fancyhdr}\usepackage{verbatim,amsthm,amstext,amssymb,color,amscd,bm}\makeatletter
\usepackage[colorlinks=true]{hyperref}\usepackage[shortlabels]{enumitem}\hypersetup{urlcolor=red, citecolor=blue}
\numberwithin{equation}{section}
\theoremstyle{plain}\newtheorem{theorem}{Theorem}[section]
\newtheorem{corollary}[theorem]{Corollary}\newtheorem{lemma}[theorem]{Lemma}\newtheorem{proposition}[theorem]{Proposition}
\theoremstyle{definition}
        
\theoremstyle{definition}
\theoremstyle{definition}
\theoremstyle{definition}
\theoremstyle{definition}
\theoremstyle{definition}
\theoremstyle{definition}\newtheorem{cl}[theorem]{Claim}

\DeclareMathOperator{\Widim}{Widim}

\DeclareMathOperator{\mdim}{mdim}
\makeatother
\newcounter{MainTheoremCounter}

\newtheorem{Maintheorem}[MainTheoremCounter]{Theorem}
\begin{document}

\title
{Minimal subshifts of prescribed mean dimension over general alphabets}

\author[Xiangtong Wang and Hang Zhao]{Xiangtong Wang and Hang Zhao}

\address{X. Wang: School of Mathematical Sciences, University of Science and Technology of China, Hefei, Anhui, 230026, People's Republic of China}

\email{wxt2020@mail.ustc.edu.cn}

\address{H. Zhao: School of Mathematics, Sichuan University, Chengdu, Sichuan, 61006, People's Republic of China}

\email{hangzhaoscu@163.com}


\begin{abstract}
Let $G$ be a countable infinite amenable group, $K$ a finite-dimensional compact metrizable space, and $(K^G,\sigma)$ the full $G$-shift on $K^G$. For any $r\in [0,\mdim(K^G,\sigma))$, we construct a minimal subshift $(X,\sigma)$ of $(K^G,\sigma)$ with $\mdim(X,\sigma)=r$. Furthermore, we construct a subshift of $([0,1]^G,\sigma)$ such that its mean dimension is $1$, and that the set of all attainable values of the mean dimension of its minimal subsystems is exactly the interval $[0,1)$.
\end{abstract}

\maketitle
\thispagestyle{empty}

\section{Introduction}

Mean  dimension is a numerical topological  invariant of dynamical systems. It was introduced by Gromov \cite{Gromov} and has been used by  Lindenstrauss and Weiss \cite{Lindenstrauss--Weiss} to prove the existence of a minimal $\mathbb{Z}$-action on a compact metrizable space which cannot be embedded into the shift on $[0,1]^{\mathbb Z}$, answering a long-standing open question raised by Auslander \cite{Auslander} in the negative. Here we say that a dynamical system can be embedded into another,  if the former is topologically conjugate to a subsystem of the latter.

Mean dimension is a powerful tool for studying dynamical systems with infinite topological  dimension or entropy.  A wonderful example is its application to the  problem of embedding dynamical systems into the shift on the Hilbert cube  $([0,1]^m)^{\mathbb Z}$. For example, when Lindenstrauss and Weiss \cite{Lindenstrauss--Weiss} solved Auslander's question, what they did is to construct a minimal dynamical system of mean dimension arbitrarily given in $[0,+\infty]$. It follows immediately that not every minimal $\mathbb Z $-action can be embedded into $(([0,1]^m)^{\mathbb Z},\sigma)$. An amazing result in this direction was due to Lindenstrauss \cite{Lin99}, he showed that for  $r\in [0,m/36)$, any minimal system of mean dimension $r$ can be embedded into $(([0,1]^m)^{\mathbb Z},\sigma)$. For any $r\ge m/2$, Lindenstrauss and Tsukamoto \cite{LT14} proved that there exists a minimal system of mean dimension $r$, which cannot be embedded into $(([0,1]^m)^{\mathbb Z},\sigma)$. As for $r\in [0,m/2)$, it is a great surprise that any minimal system of mean dimension $r$ can be embedded into $(([0,1]^m)^{\mathbb Z},\sigma)$. This result was attributed to Gutman and Tsukamoto \cite{GT20}, which improved the result of Lindenstrauss mentioned above \cite{Lin99}. 
 For its extension to $\mathbb{Z}^k$-actions, we refer to \cite{Gut15,GQS18,GT14,Gut11,GLT16,GQT19}.

In this paper, we consider mean dimension and minimal subshifts of the full $G$-shift on the product space $K^G$ (denoted by $(K^G,\sigma)$), where $G$ is a countable infinite amenable group and $K$ is a compact metrizable space.

Notice that in the case that $G=\mathbb{Z}$ and $K$ is an $m$-dimensional polyhedron, in particular $K=[0,1]^m$, Lindenstrauss and Weiss \cite{Lindenstrauss--Weiss} indeed showed that for any $r\in[0,m)$ there exists a minimal subshift of $(([0,1]^m)^\mathbb{Z},\sigma)$, whose mean dimension is exactly $r$.   In the case that $G$ is a countable amenable group containing subgroups of arbitrarily large finite index (this condition is automatically satisfied if $G$ is an infinite finitely-generated Abelian group) and $K$ is a polyhedron, for any  $r\in [0,\dim (K)]$, Coornaert and Krieger \cite{M. Coornaert and F. Krieger} constructed a  subshift  (not minimal, in general) of $(K^G,\sigma)$, whose mean dimension is $r$. 
When $G$ is a countable infinite amenable group and  $K$ is a polyhedron, for any $r\in[0,\dim (K))$, Krieger \cite{Kri09} constructed a minimal subshift of $(K^G,\sigma)$, whose mean dimension is larger than $r$.
Further, Dou \cite{Dou} improved this result. By employing tiling techniques,  he proved that for any $r\in[0,\dim (K))$ there exists a minimal subshift of $(K^G,\sigma)$ whose mean dimension  equals  $r$.
 
There are two  significant problems in this direction:
\begin{enumerate}
	\item[(1)] The first problem was posed explicitly by Dou \cite{Dou}. He asked, when $K$ is a polyhedron, whether there exists a minimal subshift of $(K^G,\sigma)$ with full mean dimension, i.e., a minimal subshift whose mean dimension equals  $\mdim(K^G,\sigma)$.
	\item[(2)] The second problem comes implicitly from the statement, asking if we may replace the alphabet (which is assumed to be any polyhedron in the above statement) with an arbitrary finite-dimensional compact metrizable space $K$.
\end{enumerate}

Problem (1) was solved by Jin and Qiao \cite{JQfull}. They successfully constructed  a minimal subshift of the Hilbert cube of full mean dimension with a delicate process, and thus, gave an affirmative answer to this problem. Although their result is stated for $\mathbb{Z}$-actions, they also pointed out that it is straightforward to generalize their result to actions of countable infinite amenable groups, one can find a detailed proof given by Yin and Xiao in \cite{Zhengyu Yin and Zubiao Xiao}.

Our goal is to address Problem (2). While our findings cover all amenable groups, it is significant to note that the corresponding results were previously unknown for the specific case of the integer group $\mathbb{Z}$.

\begin{Maintheorem}\label{maintheorem a}
	Let $G$ be a countable infinite amenable group and $K$ a finite-dimensional compact metrizable space. Then for any $0\le r<\mdim(K^G,\sigma)$ there exists a minimal subshift of $(K^G,\sigma)$, whose mean dimension is $r$. 
\end{Maintheorem}
For a finite group $G$, the conclusion of Theorem \ref{maintheorem a} is not valid.  Owing to $\mdim(X,G)={\dim(X)}/{|G|}$, the set of all  attainable values of  the mean dimension of subsystems of $(K^G,\sigma)$ is discrete.

Also note that  in contrast to $([0,1]^\mathbb{Z},\sigma)$ there is a similar result regarding the  construction  of minimal subsystems of prescribed mean dimension in a function space (so-called Bernstein space, which serves as an analogue of the Hilbert cube) by Zhao \cite{Jianjie Zhao}.

However, it remains unknown whether there exists a minimal subshift of $(K^G,\sigma)$  that achieves full mean dimension, i.e., a mean dimension equal to $\mdim(K^G,\sigma)$. It appears that the method employed in \cite{JQfull}  fails when the space $K$ significantly deviates from Euclidean spaces. Hence a natural question arises:  Is it possible to construct a dynamical system such that, for every $0\le r<\mdim(K^G,\sigma)$, there exists a minimal subsystem of mean dimension $r$, but no minimal subsystem attains  full mean dimension? For this purpose, we present a constructive example as follows:
\begin{Maintheorem}\label{maintheorem b}
	Let $G$ be a countable infinite amenable group. Then there exists a subshift $(M,\sigma)$ of $([0,1]^G,\sigma)$ such that $\mdim(M,\sigma)=1$ and that the set of all  attainable values of the mean dimension of minimal subsystems of $(M,\sigma)$ is exactly the interval $[0,1)$.
\end{Maintheorem}
Observe that if the goal is merely to construct a  dynamical system such that all of its minimal subsystems have mean dimension strictly less than the mean dimension of this system, the proof is significantly streamlined and follows directly  from Theorem \ref{maintheorem a}.

This paper is organized as follows: In Section \ref{sec:preliminaries}, we collect basic knowledge on tilings of amenable groups, subshifts and mean dimension. In Section \ref{sec:maintheorem a}, we prove Theorem \ref{maintheorem a}. In Section \ref{sec:application}, we prove Theorem \ref{maintheorem b}.

\subsection*{Acknowledgements}
The authors would like to thank Wen Huang, Qiang Huo, Lei Jin, Kairan Liu and Yixiao Qiao for their helpful suggestions. X. Wang  is partially supported by National Key R\&D Program of China (No. 2024YFA1013602, 2024YFA1013600 ) and NSFC grant (12090012, 12090010). H. Zhao is supported by NSFC grant (123B2006).

\section{Preliminaries}\label{sec:preliminaries}
\subsection{Tilings of amenable groups.}  There are several equivalent definitions of amenable groups in the literature. The one we give here is due to F\o lner \cite{Fol55}.\footnote{A detail exposition of the theory of amenable groups may be found for example in \cite{Gre69,Pat88}. }

For a group $G$, we denote by $\mathcal{F}(G)$ the collection of all nonempty finite subsets of $G$. A countable group $G$ is called \textbf{amenable} if there exists a sequence $\{F_{n}\}_{n=1}^{\infty}\subset\mathcal{F}(G)$ such that for any $g\in G$ we have $\lim_{n\to\infty}\frac{|F_{n}\triangle gF_{n}|}{|F_{n}|}=0$,
where  $|\cdot|$ denotes the cardinality of a set.  Such a sequence  $\{F_{n}\}_{n=1}^{\infty} $ is called  a \textbf{F\o lner sequence} of  the group $G$.

For $T\in\mathcal{F}(G)$ and $\epsilon>0$, we say that a subset $F$ of $G$ is \textbf{$(T,\epsilon)$-invariant} if $\frac{|B(F,T)|}{|F|}<\epsilon$,  where  $B(F,T):=\left\{g\in G:Tg\cap F\ne\emptyset,\;Tg\cap(G\setminus F)\ne\emptyset\right\}.$

The following statement can be directly proven from the definition: $\{F_{n}\}_{n=1}^{\infty}$ is a F\o lner sequence of $G$ if and only if for any $T\in\mathcal{F}(G)$ and any $\epsilon>0$, $F_{n}$ is $(T,\epsilon)$-invariant for $n$ sufficiently large if and only if for any $T\in\mathcal{F}(G)$ and any $\epsilon>0$, ${|F_{n}\triangle TF_{n}|}/{|F_{n}|}<\epsilon$ for $n$ sufficiently large.

Let $G$ be a countable infinite amenable group. We say that $\mathcal{T}$ is a \textbf{tiling} of $G$ if $\mathcal{T}\subset\mathcal{F}(G)$, $\bigcup_{T\in\mathcal{T}}T=G$ and $T\cap T'=\emptyset$ when $T\neq T'$. Every element in the tiling $\mathcal{T}$ is called a \textbf{$\mathcal{T}$-tile} (or a \textbf{tile}). A tiling $\mathcal{T}$ of $G$ is said to be \textbf{finite} if $\mathcal{T}$ can be generated by translating elements in  a finite collection  $\mathcal{S}_{\mathcal{T}}\subset\mathcal{F}(G)$, i.e., for each $T\in\mathcal{T}$ there exist $S\in\mathcal{S}_{\mathcal{T}}$ and $c\in G$ such that $Sc=T$. Every element in $\mathcal{S}_{\mathcal{T}}$ is called a \textbf{shape} of $\mathcal{T}$. For every shape $S\in\mathcal{S}_{\mathcal{T}}$ the \textbf{center} of $S$ is defined by $C(S)=\{c\in G:Sc\in\mathcal{T}\}\subset G$.
The \textbf{translation} of a tiling $\mathcal{T}$ by $g\in G$ is $\mathcal{T}g=\{Tg:T\in\mathcal{T}\}$, which is also a tiling of $G$. For $F\in\mathcal{F}(G)$ we set $\mathcal{T}|_{F}=\{T\cap F:T\in\mathcal{T}\}$. A finite tiling $\mathcal{T}$ of $G$ is called \textbf{syndetic} if for every shape $S\in\mathcal{S}_{\mathcal{T}}$ the center $C(S)$ is syndetic. 

A sequence $\{\mathcal{T}_{k}\}_{k=1}^{\infty}$ of finite tilings of $G$ is called \textbf{primely congruent} if for every $k\ge1$, $\mathcal{T}_{k}$ is a refinement of $\mathcal{T}_{k+1}$ (i.e.,  every $\mathcal{T}_{k+1}$-tile is a union of some $\mathcal{T}_{k}$-tiles) and each shape of $\mathcal{T}_{k+1}$ is partitioned by shapes of $\mathcal{T}_{k}$ in a unique way (i.e., for any two $\mathcal{T}_{k+1}$-tiles $Sc_{1}$ and $Sc_{2}$ of the same shape $S\in\mathcal{S}_{\mathcal{T}_{k+1}}$ we have ${\mathcal{T}_{k}}|_{Sc_{1}}=({\mathcal{T}_{k}}|_{Sc_{2}})c_{2}^{-1}c_{1}$).

We borrow the following two propositions in our proof.
\begin{proposition} [{\cite[Proposition 2.2]{Lei Jin and  Park Kyewon Koh and Yixiao Qiao  }}]\label{bigproportion}
	Suppose that  $\mathcal{T}$ is a finite tiling of $G$.
	Then  for any $\epsilon>0$, there exist $\delta>0$ and $K\in\mathcal{F}(G) $ such that for each $g\in G$ and each $(K,\delta)$-invariant $F\in\mathcal{F}(G)$, the union of those $\mathcal{T}g$-tiles which are contained in $F$ has proportion larger than $1-\epsilon$, namely
	$\frac{|\bigcup_{T\in\mathcal{T}g,T\subset F}T|}{|F|}>1-\epsilon$.
\end{proposition}
\begin{proposition}[{\cite[Proposition 2.3]{Lei Jin and  Park Kyewon Koh and Yixiao Qiao  }}, {\cite[Lemma 3.2]{Dou}}]  \label{manytiles}
    Suppose that  $\mathcal{T}$ is a syndetic finite tiling of $G$ and  $\mathcal{S}_{\mathcal{T}}$ is s set of shapes of  $\mathcal{T}$. Then for any $n\in\mathbb{N}$ there exist $K\in\mathcal{F}(G)$ and $\epsilon>0$ such that for every $S\in\mathcal{S}_{\mathcal{T}}$ and every $(K,\epsilon)$-invariant $F\in\mathcal{F}(G)$, $F$ contains at least $n$ $\mathcal{T}$-tiles of the shape $S$.
\end{proposition}

\subsection{Group actions}\label{subsec:groupaction}
Throughout this paper, by  a \textbf{$G$-action} we always understand a triple $(X,G,\Phi)$, where $X$ is a compact metrizable space, $G$ is a  countable infinite  discrete\footnote{Since $G$ is countable, the topology on $G$ we use here is the discrete topology.} amenable group with the identity element $e$, and $\Phi:G\times X\to X$, $(g,x)\mapsto\Phi(g,x)$ is a continuous map satisfying that
$\Phi(e,x)=x$, $\Phi(gh,x)=\Phi(g,\Phi(h,x))$, for all $x\in X$ and $g,h\in G$. Usually, $(X,G,\Phi)$ and $\Phi(g,x)$ are abbreviated to $(X,G)$ and $gx$, respectively. We say that $(Y,G)$ is a \textbf{subsystem} of $(X,G)$ if $Y$ is a closed $G$-invariant subset of $X$, under the action of $G$ restricted to $Y$.

 We say a $G$-action $(X,G)$ is \textbf{minimal} if for every $x\in X$, its \textit{orbit} $Gx:=\{gx\;:\;g\in G\}$ is dense in $X$. A subset $S$ of $G$ is called \textbf{syndetic} if there exists a finite subset $F$ of $G$ such that $G=FS$, where $FS=\{fs:f\in F,s\in S\}$. A point $x\in X$ is said to be \textbf{almost periodic} if for each neighborhood $U$ of $x$, there is a syndetic subset $S$ of $G$ such that $Sx\subset U$. The following lemma reveals a relationship between minimality and almost periodicity.

\begin{lemma}[{\cite[Chapter 1]{Auslander}}]\label{minimality}
A $G$-action $(X,G)$ is minimal if and only if $X$ is the orbit closure of an almost periodic point.
\end{lemma}
Let  $(K,d)$ be a 
compact metric space.  For each $m\in\mathbb{N}$, we equip  $K^m$ with the product topology and define a compatible metric $d_{\infty}$ on $K^{m}$  by  $$d_{\infty}\left((x_{1},x_{2},\dots,x_{m}),(y_{1},y_{2},\dots,y_{m})\right)=\max_{1\le i\le m}d(x_{i},y_{i}).$$ We also equip $K^G= \{(x_{g})_{g\in G}: x_g\in K\}$ with the product topology and define  a compatible metric $D$ on $K^G$ by 
\begin{equation}\label{metricrho}
D(x,y)=\sum_{g\in G}\alpha_{g}d(x_{g},y_{g}),\quad\forall x=(x_{g})_{g\in G},\;y=(y_{g})_{g\in G}\in K^{G},
\end{equation}
where $(\alpha_g)_{g\in G}\subset(0,+\infty)$ satisfies
$\alpha_{e}=1,\quad\sum_{g\in G}\alpha_{g}<+\infty.\footnote{Note that $G$ is countable.}$
Hence
\begin{equation}
    \label{equ:bigsmall}
    d(x_e,y_e)\le D(x,y), \quad\forall x=(x_{g})_{g\in G},\;y=(y_{g})_{g\in G}\in K^{G}.
\end{equation}

The \textbf{full $G$-shift} on $K^{G}$, denoted by $(K^G,\sigma)$, is the $G$-action  defined by
$$\sigma:G\times K^{G}\to K^{G},\quad(g,(x_{h})_{h\in G})\mapsto(x_{hg})_{h\in G}\footnote{The notation $\sigma$ may be kept in different  shifts if there is no ambiguity.}$$
for all $g\in G$  and $(x_{h})_{h\in G}\in K^{G}$. A subsystem of $(K^{G},\sigma)$ is called a \textbf{subshift}.

For $x=(x_{g})_{g\in G}\in K^{G}$ and $F\subset G$, we denote by $x|_{F}=(x_{g})_{g\in F}\in K^F$ the restriction of $x$ to $F$, and $\pi_{F}:K^{G}\to K^{F},\ x\mapsto x|_{F}$ the canonical projection map. For $q\in K$, we set $x(F,q)=\left\{g\in F:x_{g}=q\right\}\subset G$.
\subsection{Mean dimension} Here we review the definition of mean dimension. For the details, see  \cite{Gromov,Lindenstrauss--Weiss}.

Let $X$ be a compact metrizable space, $\rho$ a compatible metric on $X$, and $P$ a polyhedron.
For $\epsilon>0$, a continuous map $f:X\to P$ is called an \textbf{$\epsilon$-embedding} with respect to $\rho$, if $f(x)=f(y)$ implies $\rho(x,y)<\epsilon$, for all $x,y\in X$. Let $\Widim_{\epsilon}(X,\rho)$ be the minimum dimension of a polyhedron $P$ such that there is an $\epsilon$-embedding $f:X\to P$ with respect to the metric $\rho$. It is classically known that the topological dimension of $X$ may be recovered by $\dim(X)=\lim_{\epsilon\to0}\Widim_{\epsilon}(X,\rho)$.

Let $(X,G)$ be a $G$-action. For $F\in\mathcal{F}(G)$,
we define a  compatible metric $\rho_{F}$ on $X$ by
\begin{equation}
\label{equ:disdyn}
\rho_{F}(x,y)=\max_{g\in F}\rho(gx,gy).\end{equation}
The \textbf{mean dimension} of $(X,G)$ is defined by
$$\mdim(X,G)=\lim_{\epsilon\to0}\lim_{n\to\infty}\frac{\Widim_{\epsilon}(X,\rho_{F_{n}})}{|F_{n}|},$$
where $\{F_{n}\}_{n=1}^{\infty}$ is a F\o lner sequence of $G$. 
It is well known that the limit in the above definition always exists. The mean dimension is a topological invariant of  the $G$-action $(X,G)$, and the value $\mdim(X,G)$ is independent of the choice of the F\o lner sequence $\{F_{n}\}_{n=1}^{\infty}$ and the metric $\rho$. For example, $\mdim(([0,1]^m)^G,\sigma)=m.$

For $J\subset G,$  the \textbf{density} $\delta(J)\in [0,1]$ of $J$ in $G$ is defined by 
\begin{equation}
    \label{density}
\delta(J):=\sup_{\{F_n\}}\limsup_{n\to \infty}\cfrac{|J\cap F_n|}{|F_n|},\end{equation}
where $\{F_n\}^\infty_{n=1}$ runs over all F\o lner sequences of $G.$

We shall employ a technical lemma due to Tsukamoto \cite[Lemma 3.1]{Tsukamoto}, which applies to finite-dimensional compact metrizable spaces. 

\begin{lemma}[{\cite[Lemma 3.1]{Tsukamoto}}]\label{widimlowerbound}
	Let $(K,d)$ be a finite-dimensional compact metric space. Then there is some $\delta>0$ such that for all $m\in\mathbb{N}$ and all $0<\epsilon<\delta$ we have $\Widim_\epsilon(K^m,d_\infty)\ge m\cdot(\dim(K)-1)$.
\end{lemma}
Tsukamoto \cite{Tsukamoto} observed a dichotomy for mean dimension of the full $\mathbb{Z}$-shift on $K^\mathbb{Z}$. More precisely, he proved that either $\mdim(K^\mathbb{Z},\sigma)= \dim(K)$ or $\mdim(K^\mathbb{Z},\sigma)= \dim(K)-1$. This conclusion admits a straightforward  generalization to the setting of amenable group actions.
\begin{corollary}\label{mdimKG}
	Let $G$ be a countable infinite amenable group and $K$ a finite-dimensional compact metrizable space. Then  $$\mdim(K^G,\sigma)=\lim_{n\to\infty}\frac{\dim(K^n)}{n}.$$		
\end{corollary}
\section{A constructive proof of Theorem \ref{maintheorem a}}\label{sec:maintheorem a}

	The proof of Theorem \ref{maintheorem a} consists of three parts. 
	Part 1 is dedicated to the construction, while Part 2 and Part 3 are devoted to the argument that the $G$-action we constructed satisfies all the required conditions. The title of each part indicates the precise aim of the part.

Let us start with some necessary settings. Let $G=\{g_k:k\in\mathbb{N}\}$ be a countable infinite  discrete amenable group whose identity element is denoted by $e$. Let  $(K,d)$ be a finite-dimensional compact metric space and $ t\in (0,1)$.  
Let $\{\delta_n\}_{n=1}^\infty$ be a sequence strictly decreasing to zero and $\{K_n\}_{n=1}^\infty$ be an increasing sequence of finite $\delta_n$-dense  subsets of $K$. We take a symbol $\ast\notin K$ and set $\hat{K}=K\cup\{\ast\}$.

Our proof relies  heavily on a recent result due  to Downarowicz, Huczek and Zhang \cite{DHZ} on tilings of amenable groups, which provides the foundation for our proof. Notice that it is also feasible to obtain our result (and Dou's result \cite{Dou}) by using Ornstein and Weiss' quasi-tiling theory \cite{OW87}, but the proof will be notationally more complicated. 
\begin{theorem}[{\cite[Theorem 5.2]{DHZ}},{\cite[Theorem 3.6]{Dou}}]\label{tiling}

Let $G$ be a countable infinite  amenable group  with the identity element $e$, $\{A_n\}_{n=1}^\infty\subset\mathcal{F}(G)$ an increasing sequence with $\bigcup_{n=1}^{\infty}A_n=G$, and $\{\eta_n\}_{n=1}^\infty$ a decreasing sequence of positive numbers converging to zero. Then there exists a primely congruent sequence $\{\mathcal{T}_n\}_{n=1}^\infty$ of syndetic\footnote{The term ``syndetic'' here is equivalent to  the term ``irreducible'' in \cite{Dou}, due to \cite[Lemma 3.1]{Dou}.} finite tilings of $G$ satisfying the following conditions:
\begin{enumerate}
\item \label{propo1} $e\in S_{1,1}\subset S_{2,1}\subset\cdots\subset S_{n,1}\subset\cdots\subset\bigcup_{n=1}^{\infty}S_{n,1}=G$ and $e\in C(S_{n,1})$;
\item  \label{propo2} for every $n\in\mathbb{N}$ and every $1\le i\le \gamma_n$, $S_{n,i}$ is $(A_n,\eta_n)$-invariant;
\end{enumerate}
where for each $n\in\mathbb{N}$, $\{S_{n,i}:1\le i\le \gamma_n\}$ is the set of all shapes of $\mathcal{T}_n$.
\end{theorem}
\noindent\textbf{Part 1: Construction of $(X,\sigma)$.} 
The construction of $(X,\sigma)$ will be fulfilled by induction. To start with, let us explain the intuitive meaning of our notations shortly. 

For a positive integer $k$ we build in the $k$-th step a subshift $(X_k,\sigma)$ of $(K^G,\sigma)$, which is generated from finitely many ``blocks'' $\{B_{k,i}\}_{1\le i\le \gamma_{n_k}}$, where each $B_{k,i}$ is a closed subset of $K^{S_{n_k,i}}$ and  ${S_{n_k,i}}$ indicates the ``shape" of the block  $B_{k,i}$. When we deal with the next block $B_{k+1,j}$ in the $(k+1)$-th step, we will ``copy'' the blocks $\{B_{k,i}\}_{1\le i\le \gamma_{n_k}}$ sufficiently many times  and add some  new  restrictions to ensure that ``the proportion of freedom'' $\frac{|J_k|}{|T_k^\prime|}$ gradually decreases to the desired value $t$. 

To begin with,   we take a  decreasing sequence $\{\eta_n\}_{n=1}^\infty$ of positive numbers converging to zero and an increasing sequence $\{A_n\}_{n=1}^\infty\subset\mathcal{F}(G)$ with $\bigcup_{n=1}^{\infty}A_n=G$. By Theorem \ref{tiling}, there exists a primely congruent sequence $\{\mathcal T_n\}$ of syndetic finite tilings of $G$ with the sets of shapes $S_{\mathcal T_{n}}=\{S_{n,i}:1\le i\le \gamma_n\} $ satisfying property \eqref{propo1} and \eqref{propo2} in Theorem \ref{tiling}.

\textbf{Step 1.}
Choose $n_1\in\mathbb{N}$ sufficiently large so that for every $1\le i\le \gamma_{n_1}$ there exists $x_{1,i}\in\hat{K}^{S_{n_1,i}}$ with
$$t<\frac{|x_{1,i}(S_{n_1,i},\ast)|}{|S_{n_1,i}|}\le t+\frac{1}{|S_{n_1,i}|}.$$
Let
$$B_{1,i}=\left\{x=(x_g)_{g\in S_{n_1,i}}\in K^{S_{n_1,i}}:\;x_g=(x_{1,i})_g,\;\forall g\in S_{n_1,i}\setminus x_{1,i}(S_{n_1,i},\ast)\right\}.$$
We define $x_1\in\hat{K}^G$ by
$${x_1}|_{S_{n_1,i}c}=x_{1,i},\quad\forall1\le i\le \gamma_{n_1},\;\forall c\in C(S_{n_1,i}).$$
We set
$$X^0_1=\left\{x\in K^G:\;x|_{S_{n_1,i}c}\in B_{1,i},\;\forall1\le i\le \gamma_{n_1},\;\forall c\in C(S_{n_1,i})\right\}$$
and  the subshift of $K^G$ generated by $X^0_1$, namely,
$$X_1=\overline{GX^0_1}=\overline{\bigcup_{g\in G}gX^0_1}.$$
\textbf{Step 2.}
Applying Proposition \ref{manytiles}, we choose $l_1\in\mathbb{N}$ sufficiently large such that we can find a finite subset $R_1\subset C(S_{n_1,1})$ and $h_1\in C(S_{n_1,1})$ satisfying
$$e\in R_1,\quad h_1\notin R_1,\quad S_{n_1,1}R_1\cup S_{n_1,1}h_1\subset S_{l_1,1},\quad|R_1|=|K_1|^{|x_{1,1}(S_{n_1,1},\ast)|}.$$
We select $w_1\in\hat{K}^{S_{l_1,1}}$ by modifying ${x_1}|_{S_{l_1,1}}$ on  $${x_1}(S_{n_1,1}R_1,\ast)={x_1}(S_{n_1,1},\ast)R_1 ={x_{1,1}}(S_{n_1,1},\ast)R_1$$ such that Conditions (A.2.1), (A.2.2), and (A.2.3) are satisfied:
\begin{enumerate}
\item[(A.2.1)]
${w_1}|_{S_{n_1,1}r\setminus{x_{1,1}(S_{n_1,1},\ast)r}}={x_{1,1}}|_{S_{n_1,1}\setminus{x_{1,1}(S_{n_1,1},\ast)}}$, $\forall r\in R_1$;
\item[(A.2.2)]
${w_1}|_{x_{1,1}(S_{n_1,1},\ast)r}\in K_1^{|x_{1,1}(S_{n_1,1},\ast)|}$ $(r\in R_1)$ are pairwise distinct, i.e., \footnote{In this context, as well as in (A.k.2), we say that ${w_1}|_{x_{1,1}(S_{n_1,1},\ast)}={w_1}|_{x_{1,1}(S_{n_1,1},\ast)r}$ if and only if $(w_1)_g=(w_1)_{gr}\in K_1$ for all $g\in{x_{1,1}(S_{n_1,1},\ast)}$.}
$$\left\{{w_1}|_{x_{1,1}(S_{n_1,1},\ast)r}:\;r\in R_1\right\}=K_1^{|x_{1,1}(S_{n_1,1},\ast)|};$$
\item[(A.2.3)]
if $S_{n_1,i}c\subset S_{l_1,1}\setminus S_{n_1,1}R_1$ for some $1\le i\le \gamma_{n_1}$ and some $c\in C(S_{n_1,i})$ then
$${w_1}|_{S_{n_1,i}c}=x_{1,i}.$$
\end{enumerate}
Clearly,
$${w_1}|_{S_{n_1,1}h_1}=x_{1,1},\quad{w_1}|_{S_{n_1,1}}\in B_{1,1}\subset K^{S_{n_1,1}}.$$
We pick $n_2\in\mathbb{N}$ sufficiently large such that Conditions (B.2.1), (B.2.2), and (B.2.3) are satisfied:
\begin{enumerate}
\item[(B.2.1)]
$g_1h_1\in S_{n_2,1}$;
\item[(B.2.2)]
for every $1\le i\le \gamma_{n_2}$, there is $c_{1,i}\in C(S_{l_1,1})$ such that $S_{l_1,1}c_{1,i}\subset S_{n_2,i}$, and moreover, $c_{1,1}=e$;
\item[(B.2.3)]
for every $1\le i\le \gamma_{n_2}$, $|S_{l_1,1}|$ is negligible compared with $|S_{n_2,i}|$, more precisely,
$$\frac{|x_1(S_{n_2,i}\setminus S_{l_1,1}c_{1,i},\ast)|}{|S_{n_2,i}|}>t,
\quad\forall1\le i\le \gamma_{n_2}.$$
\end{enumerate}
For every $1\le i\le \gamma_{n_2}$, we choose $x_{2,i}\in\hat{K}^{S_{n_2,i}}$ by partially modifying ${x_1}|_{S_{n_2,i}}$ on  ${x_1}(S_{n_2,i},\ast)$ such that Conditions (C.2.1), (C.2.2), and (C.2.3) are satisfied:
\begin{enumerate}
\item[(C.2.1)]
${x_{2,i}}|_{S_{l_1,1}c_{1,i}}=w_1$;
\item[(C.2.2)]
if $S_{n_1,j}c\subset S_{n_2,i}\setminus S_{l_1,1}c_{1,i}$ for some $1\le j\le \gamma_{n_1}$ and some $c\in C(S_{n_1,j})$ then
$$(x_{2,i})_{gc}=(x_{1,j})_g,\quad\forall g\in S_{n_1,j}\setminus x_{1,j}(S_{n_1,j},\ast);$$
\item[(C.2.3)]
on the rest of coordinates in $S_{n_2,i}\setminus S_{l_1,1}c_{1,i}$,  we can replace some $\ast$'s with elements in $K$ such that there are appropriately many $\ast$'s satisfying
$$t<\frac{|x_{2,i}(S_{n_2,i},\ast)|}{|S_{n_2,i}|}\le t+\frac{1}{|S_{n_2,i}|}.$$
\end{enumerate}
Let
$$B_{2,i}=\left\{x=(x_g)_{g\in S_{n_2,i}}\in K^{S_{n_2,i}}:\;x_g=(x_{2,i})_g,\;\forall g\in S_{n_2,i}\setminus x_{2,i}(S_{n_2,i},\ast)\right\}.$$
We define $x_2\in\hat{K}^G$ by
$${x_2}|_{S_{n_2,i}c}=x_{2,i},\quad\forall1\le i\le \gamma_{n_2},\;\forall c\in C(S_{n_2,i}).$$
We set
$$X^0_2=\left\{x\in K^G:\;x|_{S_{n_2,i}c}\in B_{2,i},\;\forall1\le i\le \gamma_{n_2},\;\forall c\in C(S_{n_2,i})\right\}$$
and the subshift of $K^G$ generated by $X^0_2$, namely,
$$X_2=\overline{GX^0_2}=\overline{\bigcup_{g\in G}gX^0_2}.$$

To proceed, we assume that $x_{k-1,i}$, $B_{k-1,i}$ ($1\le i\le \gamma_{n_{k-1}}$), $x_{k-1}$, $X^0_{k-1}$,  and $X_{k-1}$ have been already generated in Step $(k-1)$. Now we generate $x_{k,i}$, $B_{k,i}$ ($1\le i\le \gamma_{n_k}$), $x_k$, $X^0_k$, and $X_k$ in Step $k$ ($k\ge2$).

\textbf{Step k.}
By Proposition \ref{manytiles}, we take $l_{k-1}\in\mathbb{N}$ large enough such that we can find a finite subset $R_{k-1}\subset C(S_{n_{k-1},1})$ and $h_{k-1}\in C(S_{n_{k-1},1})$ satisfying
$$e\in R_{k-1},\quad h_{k-1}\notin R_{k-1},\quad S_{n_{k-1},1}R_{k-1}\cup S_{n_{k-1},1}h_{k-1}\subset S_{l_{k-1},1},$$
$$|R_{k-1}|=|K_{k-1}|^{|x_{k-1,1}(S_{n_{k-1},1},\ast)|}.$$
We select $w_{k-1}\in\hat{K}^{S_{l_{k-1},1}}$  by modifying ${x_{k-1}}|_{S_{l_{k-1},1}}$ on  $${x_{k-1}}(S_{n_{k-1},1}R_{k-1},\ast)={x_{k-1}}(S_{n_{k-1},1},\ast)R_{k-1} ={x_{{k-1},1}}(S_{n_{k-1},1},\ast)R_{k-1},$$  such that Conditions (A.k.1), (A.k.2), and (A.k.3) are satisfied:
\begin{enumerate}
\item[(A.k.1)]
${w_{k-1}}|_{S_{n_{k-1},1}r\setminus{x_{k-1,1}(S_{n_{k-1},1},\ast)r}}={x_{k-1,1}}|_{S_{n_{k-1},1}\setminus{x_{k-1,1}(S_{n_{k-1},1},\ast)}}$, $\forall r\in R_{k-1}$;
\item[(A.k.2)]
${w_{k-1}}|_{x_{k-1,1}(S_{n_{k-1},1},\ast)r}\in K_{k-1}^{|x_{k-1,1}(S_{n_{k-1},1},\ast)|}$ $(r\in R_{k-1})$ are pairwise distinct, i.e.,
$$\left\{{w_{k-1}}|_{x_{k-1,1}(S_{n_{k-1},1},\ast)r}:\;r\in R_{k-1}\right\}=K_{k-1}^{|x_{k-1,1}(S_{n_{k-1},1},\ast)|};$$
\item[(A.k.3)]
if $S_{n_{k-1},i}c\subset S_{l_{k-1},1}\setminus S_{n_{k-1},1}R_{k-1}$ for some $1\le i\le \gamma_{n_{k-1}}$ and some $c\in C(S_{n_{k-1},i})$ then
$${w_{k-1}}|_{S_{n_{k-1},i}c}=x_{k-1,i}.$$
\end{enumerate}
Obviously,
$${w_{k-1}}|_{S_{n_{k-1},1}h_{k-1}}=x_{k-1,1},\quad{w_{k-1}}|_{S_{n_{k-1},1}}\in B_{k-1,1}\subset K^{S_{n_{k-1},1}}.$$
We pick $n_k\in\mathbb{N}$ sufficiently large such that Conditions (B.k.1), (B.k.2), and (B.k.3) are satisfied:
\begin{enumerate}
\item[(B.k.1)]
$g_{k-1}h_1h_2\cdots h_{k-1}\in S_{n_k,1}$;
\item[(B.k.2)]
for every $1\le i\le \gamma_{n_k}$, there is $c_{k-1,i}\in C(S_{l_{k-1},1})$ such that $S_{l_{k-1},1}c_{k-1,i}\subset S_{n_k,i}$, and moreover, $c_{k-1,1}=e$;
\item[(B.k.3)]
for every $1\le i\le \gamma_{n_k}$, $|S_{l_{k-1},1}|$ is negligible compared with $|S_{n_k,i}|$, more precisely,
$$\frac{|x_{k-1}(S_{n_k,i}\setminus S_{l_{k-1},1}c_{k-1,i},\ast)|}{|S_{n_k,i}|}>t,
\quad\forall1\le i\le \gamma_{n_k}.$$
\end{enumerate}
For every $1\le i\le \gamma_{n_k}$, we choose $x_{k,i}\in\hat{K}^{S_{n_k,i}}$ by partially modifying ${x_{k-1}}|_{S_{n_k,i}}$ on  ${x_{k-1}}(S_{n_k,i},\ast)$  such that Conditions (C.k.1), (C.k.2), and (C.k.3) are satisfied:
\begin{enumerate}
\item[(C.k.1)]
${x_{k,i}}|_{S_{l_{k-1},1}c_{k-1,i}}=w_{k-1}$;
\item[(C.k.2)]
if $S_{n_{k-1},j}c\subset S_{n_k,i}\setminus S_{l_{k-1},1}c_{k-1,i}$ for some $1\le j\le \gamma_{n_{k-1}}$ and some $c\in C(S_{n_{k-1},j})$ then
$$(x_{k,i})_{gc}=(x_{k-1,j})_g,\quad\forall g\in S_{n_{k-1},j}\setminus x_{k-1,j}(S_{n_{k-1},j},\ast);$$
\item[(C.k.3)]
on the rest of coordinates in $S_{n_k,i}\setminus S_{l_{k-1},1}c_{k-1,i}$, we can replace some $\ast$'s with elements in $K$ such that there are appropriately many $\ast$'s satisfying
$$t<\frac{|x_{k,i}(S_{n_k,i},\ast)|}{|S_{n_k,i}|}\le t+\frac{1}{|S_{n_k,i}|}.$$

\end{enumerate}
Let
$$B_{k,i}=\left\{x=(x_g)_{g\in S_{n_k,i}}\in K^{S_{n_k,i}}:\;x_g=(x_{k,i})_g,\;\forall g\in S_{n_k,i}\setminus x_{k,i}(S_{n_k,i},\ast)\right\}.$$
We define $x_k\in\hat{K}^G$ by
$${x_k}|_{S_{n_k,i}c}=x_{k,i},\quad\forall1\le i\le \gamma_{n_k},\;\forall c\in C(S_{n_k,i}).$$
We set
$$X^0_k=\left\{x\in K^G:\;x|_{S_{n_k,i}c}\in B_{k,i},\;\forall1\le i\le \gamma_{n_k},\;\forall c\in C(S_{n_k,i})\right\}$$
and the subshift of $K^G$ generated by $X^0_k$, namely,
$$X_k=\overline{GX^0_k}=\overline{\bigcup_{g\in G}gX^0_k}.$$

So far we have already generated $x_{k,i}$, $B_{k,i}$ ($1\le i\le \gamma_{n_k}$), $x_k$, $X^0_k$,  and $X_k$ in Step $k$ for all $k\in\mathbb{N}$. It follows from our construction that $\{X^0_k\}_{k=1}^\infty$  and $\{X_k\}_{k=1}^\infty$ are two decreasing sequences of non-empty closed subsets of $K^G$. Since $$S_{n_{k},1}\subset S_{n_{k+1},1}\setminus x_{k+1,1}(S_{n_{k+1},1},\ast)\subset S_{n_{m},1}\setminus x_{m,1}(S_{n_{m},1},\ast),\quad\forall k\in\mathbb{N},\;\forall m\ge k+1, $$ we deduce that 
$$x_{k+1}|_{S_{n_k,1}}=x_{k+1,1}|_{S_{n_k,1}}=w_{k}|_{S_{n_k,1}}=x_{m,1}|_{S_{n_k,1}}= x_m|_{S_{n_k,1}}\in K^{S_{n_k,1}},\quad\forall k\in\mathbb{N},\;\forall m\ge k+1.$$
Now by the fact $\bigcup_{k=1}^{\infty}S_{n_k,1}=G$ we observe that if a point $x$ belongs to the intersection $\bigcap_{k=1}^{\infty}X^0_k$ then the value $x_g\in K$ ($g\in G$) for all its coordinates must be determined eventually according to our construction. Indeed,
$x|_{S_{n_k,1}}=x_{k+1}|_{S_{n_k,1}}$ is determined in step $(k+1)$.
Thus, the intersection $\bigcap_{k=1}^{\infty}X^0_k$ is a singleton. We set
$$\bigcap_{k=1}^{\infty}X^0_k=\{z\}.$$
Finally, we let $X\subset K^G$ be the orbit closure of $z$, i.e.,
$$X=\overline{Gz}=\overline{\{gz:\;g\in G\}}.$$
Since $X$ is a closed subset of $K^G$ and is invariant under the $G$-shift, $(X,\sigma)$ becomes a subshift of $(K^G,\sigma)$, also a subsystem  of $(X_k,\sigma)$. This eventually finishes the construction of $(X,\sigma)$. Next we will check that $(X,\sigma)$ satisfies all the required properties.\\

\noindent\textbf{Part 2: Minimality of $(X,\sigma)$.}
To show that $(X,\sigma)$ is minimal, by Lemma \ref{minimality} it suffices to prove that the point $z\in X$ is almost periodic, i.e., for any $\epsilon>0$ there exists a syndetic subset $S=S_\epsilon$ of $G$ with
$$D(z,cz)<\epsilon,\quad\forall c\in S.$$

To see the later statement, we fix $\epsilon>0$ arbitrarily. Since $S_{k,1}$ is increasing over $k\in\mathbb{N}$ and eventually covers the group $G$, there exists $m\in\mathbb{N}$ such that for any two points $x=(x_g)_{g\in G}$ and $x^{\prime}=(x^{\prime}_g)_{g\in G}$  coming from $K^G$, we have
$$x|_{S_{n_m,1}}=x^\prime|_{S_{n_m,1}}\quad\text{implies}\quad D(x,x^\prime)<\epsilon.$$
Since the tiling $\mathcal{T}_{n_{m+1}}$ is syndetic, $C(S_{n_{m+1},1})$ is syndetic. By (A.k.3), (B.k.2), (C.k.1), and the definition of $z$, we have
$$z|_{S_{n_m,1}}=z|_{S_{n_m,1}c},\quad\forall c\in C(S_{n_{m+1},1})$$
i.e.,
$$z|_{S_{n_m,1}}=(cz)|_{S_{n_m,1}},\quad\forall c\in C(S_{n_{m+1},1}).$$
It follows that
$$D(z,cz)<\epsilon,\quad\forall c\in C(S_{n_{m+1},1}).$$
Thus, we end this part by taking $S=C(S_{n_{m+1},1})$.\\

\noindent\textbf{Part 3: Mean dimension of $(X,\sigma)$.}
The aim of this part is to prove
$$\mdim(X,\sigma)=t\cdot\mdim(K^G,\sigma).$$

To estimate the upper bound of $\mdim(X,\sigma)$, we will employ
Gromov's ``Pro-Mean Inequality'' (\cite[Propostion 1.9.1]{Gromov}). 
The following statement is due to Coornaert and Krieger \cite[Proposition 4.1]{M. Coornaert and F. Krieger}. 

\begin{proposition}[{\cite[Proposition 4.1]{M. Coornaert and F. Krieger}}] \label{upperbound}
Let $K$ be a finite-dimensional compact metrizable space and  $(X,\sigma)$  a subshift of $(K^{G},\sigma)$. Then one has
$$\mdim(X,\sigma)\le\liminf_{m\to\infty}\frac{\dim(\pi_{E_{m}}(X))}{|E_{m}|}$$
for any F\o lner sequence $\{E_{m}\}_{m=1}^{\infty}$ of $G$.
\end{proposition}

\begin{proof}[Upper bound for the mean dimension of $(X,\sigma)$]
Firstly, we fix $k\in\mathbb{N}$ and $\epsilon>0$ arbitrarily.
Take a F\o lner sequence $\{E_m\}_{m=1}^\infty$ of $G$. By Proposition \ref{bigproportion}, there exists $M\in\mathbb{N}$ sufficiently large such that for any $m\ge M$, the union of $\mathcal{T}_{n_k}g$-tiles which are contained in $E_m$ has proportion larger than $1-\epsilon$ for all $g\in G$, i.e.,
\begin{equation}\label{bili}
\frac{|\bigcup_{T\in\mathcal{T}_{n_k}g,T\subset E_m}T|}{|E_m|}>1-\epsilon,\quad\forall m\ge M,\;\forall g\in G.
\end{equation}
For any $m\ge M$ we divide $G$ into $L_{m,k}$ classes $Q_1,Q_2,\dots,Q_{L_{m,k}}$ such that if $g,h\in Q_i$ for some $1\le i\le L_{m,k}$ then
$$\mathcal{T}_{n_k}g|_{E_m}=\mathcal{T}_{n_k}h|_{E_m},$$
where $\mathcal{T}_{n_k}g|_{E_m}=\{Tg\cap E_m:T\in\mathcal{T}_{n_k}\}$. In fact,   $L_{m,k}$ is the total number of such different  patterns. Since $E_m$ is finite, $L_{m,k}$ is a finite number which depends on $E_m$ and the shapes of $\mathcal{T}_{n_k} $. For each $1\le i\le L_{m,k}$ we take $q_i\in Q_i$, then there exists $j_i=j_i(m,k)\in\mathbb{N}$ such that
$$\mathcal{T}_{n_k}q_i|_{E_m}=\left\{S_{n_k,p_1}c_1q_i,S_{n_k,p_2}c_2q_i,\dots,S_{n_k,p_{j_i}}c_{j_i}q_i,\;A_i\right\}$$
for some $1\le p_l\le \gamma_{n_k}$, $c_l\in C(S_{n_k,p_l})$ ($1\le l\le j_i$) and some $A_i=A_i(m,k)\subset E_m$ with $|A_i|/|E_m|<\epsilon$. Since 
$$\pi_{E_m}({q_iX^0_k})\subset B_{k,p_1}\times B_{k,p_2}\times\cdots\times B_{k,p_{j_i}}\times K^{A_i},\quad\forall m\ge M $$
is a closed set,  by the construction of ${X_k}$, it follows  that 
\begin{align*}
\pi_{E_m}({X_k})
&\subset\overline{\bigcup_{g\in G}\pi_{E_m}({gX^0_k})}
&\subset\bigcup_{1\le i\le L_{m,k}}B_{k,p_1}\times B_{k,p_2}\times\cdots\times B_{k,p_{j_i}}\times K^{A_i},\quad\forall m\ge M.
\end{align*}
Thus, we have
\begin{align*}
\frac{\dim(\pi_{E_m}({X_k}))}{|E_m|}
&\le\max_{1\le i\le L_{m,k}}\frac{\dim(B_{k,p_1}\times B_{k,p_2}\times\cdots\times B_{k,p_{j_i}}\times K^{A_i})}{|E_m|}\\
&\le\max_{1\le i\le L_{m,k}}\frac{\dim K^{\sum_{1\le l\le j_i}|x_{k,p_l}(S_{n_k,p_l},\ast)|}+|A_i|\cdot\dim(K)}{|E_m|}.
\end{align*}
Let $t_{m,k}:=\max_{1\le i\le L_{m,k}}\sum_{1\le l\le j_i}|x_{k,p_l}(S_{n_k,p_l},\ast)|$. By (C.k.3) we have that
$$\max_{1\le i\le L_{m,k}}t\cdot\sum_{1\le l\le j_i}|S_{n_k,p_l}| < t_{m,k}\le\max_{1\le i\le L_{m,k}}\sum_{1\le l\le j_i}(t\cdot|S_{n_k,p_l}|+1),$$
then  for each $k\in \mathbb{N}$, we have 
 $$\lim_{m\to\infty}t_{m,k}=\infty\quad\text{and}\quad \frac {t_{m,k}} {|E_{m}|}\le t+\frac{1}{\min\{|S_{n_k,j}|:1\le j\le \gamma_{n_k}\}}.$$
Thus, by \eqref{bili} we obtain
\begin{align*}
\frac{\dim(\pi_{E_m}({X_k}))}{|E_m|}
&\le\frac{\dim(K^{t_{m,k}})}{t_{m,k}}\cdot\frac {t_{m,k}} {|E_m|}+\epsilon\cdot \dim(K)\\
&\le\frac{\dim(K^{t_{m,k}})}{t_{m,k}}\cdot\left(t+\frac{1}{\min\{|S_{n_k,j}|:1\le j\le \gamma_{n_k}\}}\right) +\epsilon\cdot \dim(K)
\end{align*}
for all $m\ge M$. By Proposition \ref{upperbound} and  Corollary \ref{mdimKG}, we obtain
$$\mdim({X_k},\sigma)\le \mdim(K^G,\sigma)\cdot(t+\frac{1}{\min\{|S_{n_k,j}|:1\le j\le \gamma_{n_k}\}}) +\epsilon\dim(K).$$
Since $k\in\mathbb{N}$ and $\epsilon>0$ are arbitrary, and since $\mdim(X,\sigma)\le\mdim({X_k},\sigma)$ for all $k\in\mathbb{N}$, it follows that
$$\mdim(X,\sigma)\le\lim_{k\to\infty}\left(\mdim(K^G,\sigma)\cdot(t+\frac{1}{\min\{|S_{n_k,j}|:1\le j\le \gamma_{n_k}\}}) \right)= t\cdot \mdim(K^G,\sigma).$$
This finishes the proof of the upper bound of $\mdim(X,\sigma)$.
\end{proof}

\medskip
 
We now turn to the lower bound of $\mdim(X,\sigma)$. In order to show 
$$\mdim(X,\sigma)\ge t\cdot\mdim(K^G,\sigma),$$
we need more preparations than the case when $K=P$ is a polyhedron.

The following proposition generalizes a useful tool for estimating the lower bound of the mean dimension, as presented  in \cite[Proposition 2.8]{Kri09}, from  polyhedra to  finite-dimensional compact metric spaces. 
We note  that this method is originally due to Lindenstrauss and Weiss for $P^{\mathbb{Z}}$, where $P$ is a polyhedron \cite[Proposition 3.3]{Lindenstrauss--Weiss}. To obtain Proposition \ref{pro:lowerbound}, we employ a recent tool due to Tsukamoto \cite{Tsukamoto} to overcome the difficulties in both topology and dimension theory arising from general alphabets.
\begin{proposition}
	\label{pro:lowerbound}
	Let G be a countable amenable group. Let $(X,\sigma)\subset (K^G,\sigma)$ be a  subshift, where $K$ is a finite-dimensional compact metrizable space.  Suppose that there exist $x'=(x'_g)_{g\in G}\in X$ and a subset $J\subset G$ satisfying the following condition
	$$\pi_{G\setminus J}(x)=\pi_{G\setminus J}(x')\Rightarrow x\in X,$$
	for all $x\in K^G.$
	
	Then we have
	$$\mdim(X,\sigma)\ge \delta(J)\mdim(K^G,\sigma).\footnote{Recall that the density $\delta(J)$ is defined in \eqref{density}.}$$
	
\end{proposition}
\begin{proof}
 Consider a compatible metric $d$ on $K$.  For every $m\in\mathbb{N}$, the closed subset $X^m\subset K^m$ inherits  the metric from $K^m$, which is induced by $d_{\infty}$ as in \eqref{metricrho} and denoted by $D_{\infty}$.
Let $\{F_n\}_{n=1}^\infty$ be a F\o lner sequence of $G$ and fix $n\in\mathbb{N}$. Define $J_n=J\cap F_n$. Consider the embedding  $f_n:(K^m)^{J_n}\to (K^m)^G$ which associates to each $u=(u_g)_{g\in J_n}\in(K^m)^{J_n}$ the element $u'=(u'_g)_{g\in G}\in (K^m)^G$ defined by 
$$u^\prime_g=
\begin{cases}
u_g,&\quad g\in J_n,\\
(\underbrace{x_g',\cdots,x_g'}_{m-\text{times}}),&\quad g\in G\setminus J_n.
\end{cases}$$
The properties satisfied by $x'$ and $J$ imply  $f_n(u)\in X^m$ for all $u\in (K^m)^{J_n}$.
	Notice that  $(D_\infty)_{J_n}$ is a compatible metric on $(K^m)^G$ defined as in \eqref{equ:disdyn}. By \eqref{equ:bigsmall}, we deduce
	$$d_{\infty}(u,v)\le (D_\infty)_{J_n}(f_n(u),f_n(v)),\quad\forall u,v\in (K^m)^{J_n}.$$
	Combining this with the fact that $J_n\subset F_n,$ we deduce  
	$$d_{\infty}(u,v)\le (D_\infty)_{J_n}(f_n(u),f_n(v)) \le (D_\infty)_{F_n}(f_n(u),f_n(v)),\quad\forall u,v\in (K^m)^{J_n}.$$
	Hence we have that for any $\epsilon>0$ and $n\in \mathbb N,$ 
	$$\Widim_\epsilon\left(X^m,(D_\infty)_{F_n}\right)\ge\Widim_\epsilon\left((K^m)^{J_n},d_{\infty}\right).$$
	By Lemma \ref{widimlowerbound} we know that  there exists $\delta>0$, such that for any $0<\epsilon<\delta$  we deduce
	$$\Widim_\epsilon\left((K^m)^{J_n},d_{\infty}\right)\ge |J_n|\cdot\left(\dim(K^m)-1\right) ,\quad\forall n\in\mathbb{N}.$$
	Then we have 
	\begin{align*}
	\mdim(X,\sigma)
	&\ge\frac{\mdim(X^m,\sigma)}{m}\\
	&=\frac{1}{m}\cdot\lim_{\epsilon\to0}\lim_{n\to\infty}\frac{\Widim_\epsilon\left(X^m,(D_\infty)_{F_n}\right)}{|F_n|}\\
	&\ge \frac{1}{m}\cdot\lim_{\epsilon\to0}\limsup_{n\to\infty}\frac{\Widim_\epsilon\left((K^m)^{J_n},d_{\infty}\right)}{|F_n|}\\
	&\ge \limsup_{n\to\infty}\frac{|J_n|}{|F_n|}\cdot \frac{\dim(K^m)-1}{m}.\end{align*}
	Since $\{F_n\}_{n=1}^\infty$ is arbitrary,  we deduce  $$\mdim(X,\sigma)\ge\delta(J)\cdot\frac{\dim(K^m)-1}{m}.$$
	Sine $m\in\mathbb{N}$ is arbitrary,  by Corollary \ref{mdimKG} we finally conclude that
	$$\mdim(X,\sigma)\ge \delta(J)\cdot\mdim(K^G,\sigma).$$
\end{proof}

In order to establish the lower bound,  we need the following preparations.

For any $ k\in\mathbb{N}$ and $ m\ge k+1$,  we list some properties of our constructions according to (A.k.3), (B.k.2), and (C.k.1):
\begin{enumerate}
  \item[(a)]
       $S_{n_{k},1}h_k\cdots h_{m-1}\subset S_{n_{m},1}$;
  \item[(b)]
	$x_{k,1}(S_{n_{k},1},\ast)h_k\cdots h_{m-1}\subset x_{m,1}(S_{n_{m},1},\ast)$.
  \item[(c)]$(S_{n_{k},1}\setminus x_{k,1}(S_{n_{k},1},\ast))h_k\cdots h_{m-1}\subset  S_{n_{m},1}\setminus x_{m,1}(S_{n_{m},1},\ast).$

\end{enumerate}
Set
$$T_1^\prime=S_{n_1,1},\quad T_k^\prime=S_{n_k,1}h_{k-1}^{-1}\cdots h_1^{-1},\quad\forall k\ge2.$$
Since $\{S_{n_k,1}\}_{k=1}^\infty$ is a F\o lner sequence of $G$, so is the sequence $\{T_k^\prime\}_{k=1}^\infty$. According to (a), we have
$$S_{n_k,1}
\subset S_{n_{k+1},1}h_k^{-1},\quad\forall k\in\mathbb{N}.$$
It follows that
$$T_k^\prime=S_{n_k,1}h_{k-1}^{-1}\cdots h_1^{-1}\subset S_{n_{k+1},1}h_k^{-1}h_{k-1}^{-1}\cdots h_1^{-1}=T_{k+1}^\prime,\quad\forall k\in\mathbb{N}.$$
By (B.k.1),
$$g_k\in S_{n_{k+1},1}h_k^{-1}\cdots h_1^{-1}=T_{k+1}^\prime,\quad\forall k\in\mathbb{N}.$$
Therefore $\{T_k^\prime\}_{k=1}^\infty$ is an increasing F\o lner sequence of $G$ with $\bigcup_{k=1}^{\infty}T_k^\prime=G$.
Set
$$J_1=x_{1,1}(S_{n_{1},1},\ast),\quad
J_k=x_{k,1}(S_{n_{k},1},\ast)h_{k-1}^{-1}\cdots h_1^{-1},\quad\forall k\ge2.$$ Then $ J_k\subset T_k^\prime$. It follows from (b) that
$$x_{k,1}(S_{n_{k},1},\ast)h_k\subset x_{k+1,1}(S_{n_{k+1},1},\ast),\quad\forall k\in\mathbb{N}.$$
Thus,
$$J_k\subset J_{k+1},\quad\forall k\in\mathbb{N}.$$
Let
\begin{equation}\label{J}
	J=\bigcup_{k=1}^{\infty}J_k.
\end{equation}

The property of $X$ below is crucial for the proof  of lower bound of $\mdim(X,\sigma)$.
\begin{proposition}\label{projection}
There exists  $\omega \in K^{G\setminus J}$, such that for any $u\in K^J$ we can find $x=x(u)\in X$ such that
$$x|_J=u,\quad x|_{G\setminus J}=\omega.$$
In other words, we have
$$\pi_{G\setminus J}(x)=\omega\Rightarrow x\in X$$
    for all $x\in K^G.$
\end{proposition}
  We now prove the lower bound of $\mdim(X,\sigma)$ assuming Proposition \ref{projection}.
\begin{proof}[Lower bound for the mean dimension of $(X,\sigma)$]
Since $\{T_k'\}_{k=1}^\infty$ is a F\o lner sequence and $J_k= J\cap T_k'$ for all $k\in \mathbb{N}$, we have
$$\delta(J)\ge \limsup_{k\to \infty}\cfrac{|J\cap T_k'|}{|T_k'|}= \limsup_{k\to \infty}\cfrac{|J_k|}{|T_k'|}.$$
It follows from (C.k.3) that
$$t<\frac{|x_{k,1}(S_{n_k,1},\ast)|}{|S_{n_k,1}|}=\frac{|J_k|}{|T_k^\prime|}\le t+\frac{1}{|T_k^\prime|}.$$
Since $k\in\mathbb{N}$ is arbitrary, we obtain
$$\lim_{k\to \infty}\cfrac{|J_k|}{|T_k'|}=t.$$
Combining  this with Proposition \ref{pro:lowerbound} and  Proposition \ref{projection}, we have
$$\mdim(X,\sigma)\ge \delta(J)\cdot\mdim(K^G,\sigma)\ge \lim_{k\to \infty}\cfrac{|J_k|}{|T_k'|}\cdot \mdim(K^G,\sigma)=t\cdot\mdim(K^G,\sigma).$$

\end{proof}

It remains to prove Proposition \ref{projection}.

\begin{proof}[Proof of Proposition \ref{projection}]

Notice that by (c), $T_k^\prime\setminus J_k$ is increasing over $k\in\mathbb{N}$. Moreover, $G\setminus J=\bigcup_{k=1}^{\infty}(T_k^\prime\setminus J_k)$. For arbitrary $g\in G\setminus J$, choose the  smallest  $l(g)\in\mathbb{N}$ satisfying
$$g\in T_k^\prime\setminus J_k = S_{n_k,1}h_{k-1}^{-1}\cdots h_1^{-1}\setminus x_{k,1}(S_{n_k,1},\ast)h_{k-1}^{-1}\cdots h_1^{-1},\quad\forall k\ge l(g).$$
Define $$\omega=(\omega_g)_{g\in G\setminus J}=(z_{gh_1\cdots h_{l(g)-1}})_{g\in G\setminus J}\in K^{G\setminus J}.$$ Fix  $g\in G\setminus J$,  it follows  $gh_1\cdots h_{l(g)-1}\in S_{n_{l(g)},1}\setminus x_{l(g),1}(S_{n_{l(g)},1},\ast)$. For any  $k\ge l(g)$, we have 
 $gh_1\cdots h_{k-1}=(gh_1\cdots h_{l(g)-1})\cdot h_{l(g)}\cdots h_{k-1}$. Combining this with (A.k.3) and $z\in \bigcap_{k=1}^{\infty}X^0_k$,  we deduce that
$$z_{gh_1\cdots h_{l(g)-1}}=z_{gh_1\cdots h_{k-1}}.$$
For arbitrary  $u\in K^J$, we can define $x=x(u)\in  K^G $ by
$$x_g=\begin{cases}
\omega_g,&\quad g\in G\setminus J,\\
u_g,&\quad g\in J.
\end{cases}$$ 

We will show that $x\in X=\overline{\{gz:g\in G\}}$. In fact, it is enough to show that for any $\epsilon>0$, there exists  $c\in G$ such that
$D(x,cz)<\epsilon.$ 

To this end, we fix $\epsilon>0$ arbitrarily. Since $\{\delta_n\}_{n=1}^\infty$ is a sequence strictly decreasing to zero and $ T_k^\prime$ is increasing over $k\in\mathbb{N}$ which eventually covers the group $G$, there exist  $s=s(\epsilon)\in\mathbb{N}$ and $p=p(s,\epsilon)\ge\max_{g\in T_s^\prime\setminus J_s}\{l(g),s+1\}$ such that 
$$d_\infty(x^\prime|_{T^\prime _{s}},x^{\prime\prime}|_{T^\prime _{s}} )<\delta_{p} \quad\text{implies}\quad D(x,x^\prime)<\epsilon. $$
By (b), we have $x_{s,1}(S_{n_{s},1},\ast)h_s\cdots h_{p-1}\subset x_{p,1}(S_{n_{p},1},\ast)$.  Combining this with (A.k.2) and $z\in \bigcap_{k=1}^{\infty}X^0_k$, we deduce that 
$$K_{p}^{x_{s,1}(S_{n_s,1},\ast)}\subset \left\{{x_{p+1,1}}|_{x_{s,1}(S_{n_s,1},\ast)h_s\cdots h_{p-1}r}:\;r\in R_{p}\right\}=\left\{z|_{x_{s,1}(S_{n_s,1},\ast)h_s\cdots h_{p-1}r}:\;r\in R_{p}\right\} .$$
There exists $r_0\in R_{p}$, such that $$d_\infty(x|_{J_s}, z|_{J_sh_1\cdots h_{p-1}r_0})=d_\infty(x|_{J_s}, z|_{x_{s,1}(S_{n_s,1},\ast)h_s\cdots h_{p-1}r_0})<\delta_{p}.$$
Notice that $z|_{x_{s,1}(S_{n_s,1},\ast)h_s\cdots h_{p-1}r_0}=(h_1\cdots h_{p-1}r_0z)|_{J_s}$,  it follows $$d_\infty(x|_{J_s},(h_1\cdots h_{p-1}r_0z)|_{J_s})<\delta_{p}.$$
By (c), we have $(S_{n_s,1}\setminus x_{s,1}(S_{n_{s},1},\ast))h_s\cdots h_{p-1}\subset S_{n_{p},1}\setminus  x_{p,1}(S_{n_{p},1},\ast)$. Combining this with (A.k.1) and $z\in \bigcap_{k=1}^{\infty}X^0_k$, we deduce that $$z|_{(S_{n_s,1}\setminus x_{s,1}(S_{n_{s},1},\ast))h_s\cdots h_{p-1}r_0}=z|_{(S_{n_s,1}\setminus x_{s,1}(S_{n_{s},1},\ast))h_s\cdots h_{p-1}}.$$ 
Notice that $z|_{(S_{n_s,1}\setminus x_{s,1}(S_{n_{s},1},\ast))h_s\cdots h_{p-1}r_0}  =(h_1\cdots h_{p-1}r_0z)|_{T^\prime _s\setminus J_s}$ and $z|_{(S_{n_s,1}\setminus x_{s,1}(S_{n_{s},1},\ast))h_s\cdots h_{p-1}}= (h_1\cdots h_{p-1}z)|_{T^\prime _s\setminus J_s} $ ,  it follows $$(h_1\cdots h_{p-1}r_0z)|_{T^\prime _s\setminus J_s}= (h_1\cdots h_{p-1}z)|_{T^\prime _s\setminus J_s}. $$
Therefore,
$$x|_{T_s^\prime\setminus J_s}=\omega|_{T_s^\prime\setminus J_s}= (h_1\cdots h_{p-1}z)|_{T^\prime _s\setminus J_s}=(h_1\cdots h_{p-1}r_0z)|_{T_s^\prime\setminus J_s}.$$
It follows that
$$d _{\infty}(x|_{T_s^\prime},(h_1\cdots h_{p-1}r_0z)|_{T_s^\prime})<\delta_{p},$$ which implies  $$ D(x,h_1\cdots h_{p-1}r_0z)<\epsilon.$$  Taking $c=h_1\cdots h_{p-1}r_0$,  thus this completes the proof.
\end{proof}
\section{Proof of Theorem \ref{maintheorem b}}\label{sec:application}
 In this section, we will prove Theorem \ref{maintheorem b}  by constructing a subshift $(M,\sigma)$ of $([0,1]^G,\sigma).$

\begin{proof}[Proof of Theorem \ref{maintheorem b}] 
Let $\{\delta_n\}_{n=1}^\infty$ be a sequence strictly decreasing to zero with $0<\delta_n<\frac{1}{2^{n+1}}$ and $\{r_n\}_{n=1}^\infty$ be a sequence strictly increasing to 1 with $0<r_n<1$. For each $n\in\mathbb{N}$, let $I_n=[1-\frac{1}{2^{n-1}},1-\frac{1}{2^{n}}-\delta_n]$. 

Fix $n\in\mathbb{N}$, we now proceed to construct a subshift $(Y_n,\sigma)$ that corresponds to $r_n$ as follows.  Let $\{\mathcal{T}_k\}_{k=1}^{\infty}$ be the sequence of tilings in Theorem \ref{tiling}, then  $\{S_{k,1}\}_{k=1}^{\infty}$ is an increasing F\o lner sequence of $G$ with $\bigcup_{k=1}^{\infty}S_{k,1}=G$. We take a symbol $\ast\notin I_n$ and set $\hat{I_n}=I_n\cup\{\ast\}$, then choose $n_1\in\mathbb{N}$ sufficiently large so that for every $1\le i\le \gamma_{n_1}$ there is $y_{i}=y_{i}(n)\in\hat{I_n}^{S_{n_1,i}}$ with
\begin{equation}
\label{midu}r_n<\frac{|y_{i}(S_{n_1,i},\ast)|}{|S_{n_1,i}|} < 1.\end{equation}
	Let
	$$B_{1,i}=\left\{y=(y_g)_{g\in S_{n_1,i}}\in K^{S_{n_1,i}}:\;y_g=(y_{i})_g,\;\forall g\in S_{n_1,i}\setminus y_{1,i}(S_{n_1,i},\ast)\right\}.$$
	We define $y_1\in\hat{K}^G$ by
	$${y_1}|_{S_{n_1,i}c}=y_{i},\quad\forall1\le i\le \gamma_{n_1},\;\forall c\in C(S_{n_1,i}),$$
	then we set
	$$Y_n^0=\left\{y\in I_n^G:\;y|_{S_{n_1,i}c}\in B_{1,i},\;\forall1\le i\le \gamma_{n_1},\;\forall c\in C(S_{n_1,i})\right\}$$
	and  the subshift of $K^G$ generated by $Y^0_n$, namely, 
	$$Y_n=\overline{GY_n^0}=\overline{\bigcup_{g\in G}gY_n^0}.$$

Next, we define
$$ M=\cup_{n=1}^{\infty}Y_n\cup\{1\}^{G}.$$
Hence $M$ is a shift-invariant closed set. For any $r\in[0,1)$,  there exists $n\in \mathbb{N}$ such that  $r_n>r$, now we claim that 
\begin{cl}
\label{Uni}

There exists a  minimal subsystem $(X,\sigma)$ of $(Y_n,\sigma)$ with a mean dimension equal to $r$. 
    
\end{cl}
\begin{proof}
The construction of  minimal dynamical system $(X,\sigma)$ will be fulfilled by induction:

\textbf{Step 1.} Since $r_n>r$, by \eqref{midu} we  can define $x_{1,i}=x_{1,i}(n,r)\in\hat{I_n}^{S_{n_1,i}}$ by replacing some $*$'s in $y_{i}$ with element in $I_n$
 such that there are appropriately many $*$'s satisfying
 $$r<\frac{|x_{1,i}(S_{n_1,i},\ast)|}{|S_{n_1,i}|}\le r+\frac{1}{|S_{n_1,i}|}.$$
 Let
$$B^n_{1,i}=\left\{x=(x_g)_{g\in S_{n_1,i}}\in I_n^{S_{n_1,i}}:\;x_g=(x_{1,i})_g,\;\forall g\in S_{n_1,i}\setminus x_{1,i}(S_{n_1,i},\ast)\right\}.$$
We define $x_1\in\hat{I_n}^G$ by
$${x_1}|_{S_{n_1,i}c}=x_{1,i},\quad\forall1\le i\le \gamma_{n_1},\;\forall c\in C(S_{n_1,i}).$$
We set
$$X^0_1=X^0_{1}(n,r)=\left\{x\in I_n^G:\;x|_{S_{n_1,i}c}\in B^n_{1,i},\;\forall1\le i\le \gamma_{n_1},\;\forall c\in C(S_{n_1,i})\right\}$$

and
$$X_{1}=X_1(n,r)=\overline{GX^0_{1}}=\overline{\bigcup_{g\in G}gX^0_{1}}.$$

\textbf{Step 2.} Let $t=r,$ $K=I_n$ and $X_{1}=X_{1}(n,r)$ 
in the proof of Theorem \ref{maintheorem a} (Step 2). We can find $n_2\in \mathbb N$, $x_2\in \hat{I}_n^{G}$,  
$X_{2}=X_2(n,r)$ and $x_{2,i}=x_{2,i}(n,r)$ satisfies $$ r<\frac{|x_{2,i}(S_{n_2,i},\ast)|}{|S_{n_2,i}|}\le r+\frac{1}{|S_{n_2,i}|}.$$

\textbf{Step k.} Let $t=r,$ $K=I_n$, and $X_k=X_{k}(n,r)$  in the proof of Theorem \ref{maintheorem a} (Step k). We can find $n_k\in \mathbb N$, $x_k\in \hat{I}_n^{G}$,  
$X_{k}=X_k(n,r)$ and $x_{k,i}=x_{k,i}(n,r)$ satisfies $$ r<\frac{|x_{k,i}(S_{n_k,i},\ast)|}{|S_{n_k,i}|}\le r+\frac{1}{|S_{n_k,i}|}.$$

By using the same method again  as in  the proof of  Theorem \ref{maintheorem a},  we can find $z=z(n,r)\in I_n^{G}$
 satisfying
 $$\bigcap_{k=1}^{\infty} X_{k}^0=\{z\},$$
then set $$X=\overline{Gz}=\overline{\{gz:\;g\in G\}}.$$
Using the same method again  as in  the proof of  Theorem \ref{maintheorem a}, we can prove that $(X,\sigma)$ is a minimal system of mean dimension $r$.
\end{proof}

Since each minimal subshift of $(M,\sigma)$ will either be contained in some $Y_n$ or be a singleton, the  attainable values of  the mean dimension for  minimal subsystems  will never reach $1$. Combining this with the fact that $r_n\le\mdim(M,\sigma)\le 1$ for arbitrary $n \in\mathbb{N}$, it follows that $$\mdim(M,\sigma)=1.$$ This concludes the proof of Theorem \ref{maintheorem b}.	
\end{proof} 

It is worth noting that in our construction, the alphabet $[0,1]$, the condition $\mdim(M,\sigma)=1$, and the interval of attainable values  $[0,1)$ can effectively be replaced by any polyhedron $P$, $\mdim(M,\sigma)=\dim(P)$, and $[0,\dim(P))$, respectively.

\end{document}